\documentclass[11pt]{article}
\usepackage{}
\usepackage{amsfonts}
\usepackage{amsthm}
\usepackage{amsmath}
\usepackage{amssymb}
\usepackage{mathrsfs}
\usepackage{amscd}
\usepackage{mathdots}
\usepackage[all]{xy}

\newcommand{\V}{\mathcal {V}}
\newcommand{\ES}{\mathscr{S}}

\newcommand{\dr}{\mathrm{dR}}
\newcommand{\sdr}{s_{\dr}}

\newcommand{\Isom}{\mathbf{Isom}}

\newcommand{\Sh}{\mathrm{Sh}}
\newcommand{\Hdr}{\mathrm{H}^1_{\mathrm{dR}}}

\normalsize \relpenalty=10000 \binoppenalty=10000

\begin{document}
\newtheorem{theorem}[subsection]{Theorem}
\newtheorem{lemma}[subsection]{Lemma}
\newtheorem{proposition}[subsection]{Proposition}
\newtheorem{corollary}[subsection]{Corollary}
\theoremstyle{definition}
\newtheorem{definition}[subsection]{Definition}
\theoremstyle{definition}
\newtheorem{construction}[subsection]{Construction}
\theoremstyle{definition}
\newtheorem{notations}[subsection]{Notations}
\theoremstyle{definition}
\newtheorem{asp}[subsection]{Assumption}
\theoremstyle{definition}
\newtheorem{set}[subsection]{Setting}
\theoremstyle{remark}
\newtheorem{remark}[subsection]{Remark}
\theoremstyle{remark}
\newtheorem{example}[subsection]{Example}

\makeatletter
\newenvironment{subeqn}{\refstepcounter{subsection}
$$}{\leqno{\rm(\thesubsection)}$$\global\@ignoretrue}
\makeatother

\title{Remarks on Ekedahl-Oort stratifications}
\author{Chao Zhang}\date{}
\maketitle \setcounter{section}{-1} \setcounter{tocdepth}{1}
\textbf{Abstract:}
This short paper is a continuation of the author's Ph.D thesis, where Ekedahl-Oort strata are defined and studied for Shimura varieties of Hodge type. The main results here are as follows.

1. The Ekedahl-Oort stratification is independent of the choices of symplectic embeddings.

2. Under certain reasonable assumptions, there is certain functoriality for Ekedahl-Oort stratifications with respect to morphisms of Shimura varieties.

\

\

\section[Introduction]{Introduction}

Let $(G,X)$ be a Shimura datum of Hodge type, and $\Sh_K(G,X)$ be the Shimura variety attached to a compact open subgroup $K\subseteq G(\mathbb{A}_f)$. We assume that $K=K_pK^p$, where $K_p$ is hyperspecial, i.e. there is a reductive group $\mathcal {G}/\mathbb{Z}_p$ such that $\mathcal {G}_{\mathbb{Q}_p}=G_{\mathbb{Q}_p}$ and that $K_p=\mathcal {G}(\mathbb{Z}_p)$. By works of Deligne, $\Sh_K(G,X)$ is defined over a number field $E$. Let $v$ be a place of $E$ over $p$, then Kisin proved in \cite{CIMK} that $\Sh_K(G,X)$ has a smooth model $\ES_K(G,X)$ over $O_{E,(v)}$. Moreover, $\ES_K(G,X)$ is uniquely determined by the Shimura datum in the sense that $\varprojlim_{K^p}\ES_K(G,X)$ satisfies a certern extension property (see \cite{CIMK} 2.3.7 for a precise statement). 

Ekedahl-Oort stratifications for good reductions of Shimura varieties of Hodge type were defined and studied in \cite{EOZ} using \cite{CIMK} and \cite{zipaddi}. Let $\kappa=O_{E,(v)}/(v)$ and $\ES_0$ (resp. $G_0$) be the special fiber of $\ES_K(G,X)$ (resp. $\mathcal {G}$). The Shimura datum determines a cocharacter $\mu:\mathbb{G}_{m,\kappa}\rightarrow G_{0,\kappa}$ which is unique up to $G_{0}(\kappa)$-conjugacy. We constructed in \cite{EOZ} a morphism $\zeta:\ES_0\rightarrow G_0\texttt{-Zip}_\kappa^{\mu}$, where $G_0\texttt{-Zip}_\kappa^{\mu}$ is the stack of $G_0$-zips of type $\mu$ (see \cite{zipaddi} or $\S$ 1.2 of this paper). Fibers of $\zeta$ are Ekedahl-Oort strata. Note that to construct $\zeta$, we need to fix a symplectic embedding. 

There are many basic questions that were not solved in \cite{EOZ}. Here we mention two of them.

1. Whether the Ekedahl-Oort stratification (namely, the morphism $\zeta$) is independent of the choices of symplectic embeddings.

2. How to study behavior of stratifications under morphisms of Shimura varieties.

The motivation for the first question is the observation that both the reduction of the Shimura variety and the ``list'' of Ekedahl-Oort strata are independent of choices of symplectic embeddings. While the second one is a question that can not be more natural. 

The first question is solved by the following theorem. 
\begin{theorem}
The morphism $\zeta$ is uniquely determined by $(G,X)$ and $\mu$, and hence independent of choices of symplectic embeddings.
\end{theorem}
The first section is devoted to a proof of the above result. 

The second question is too general and too inexplicit to study, so we raise the following question. Let $f:(G,X)\rightarrow (G',X')$ be a morphism of Shimura data of Hodge type. Let $E$ and $E'$ be their reflex fields. Then $E\supseteq E'$. Let $K\subseteq G(\mathbb{A}_f)$ and $K'\subseteq G'(\mathbb{A}_f)$ be such that $K_p$ and $K'_p$ are hyperspecial. Assume that $f(K)\subseteq K'$, then there is a morphism $f: \Sh_K(G,X)\rightarrow \Sh_{K'}(G',X')_E$. Let $v'$ be a place of $E'$ over $p$ with residue field $\kappa'$ and $v$ be a place of $E$ over $v'$ with residue field $\kappa$, then there is a morphism $\ES_K(G,X)\rightarrow \ES_{K'}(G',X')_{O_{E,(v)}}$ extending $f$. Still write $f$ for the morphism on special fibers $\ES_{0,K}(G,X)\rightarrow \ES_{0,K'}(G',X')_{\kappa}$. Let $G_0$ (resp. $G'_0$) be the reduction of $G$ (resp. $G'$), and let $\mu$ (resp. $\mu'$) be the cocharacter unique up to conjugacy. Then there is a morphism $\zeta:\ES_{0,K}(G,X)\rightarrow G_0\texttt{-Zip}_\kappa^{\mu}$ (resp. $\zeta':\ES_{0,K'}(G',X')\rightarrow G'_0\texttt{-Zip}_{\kappa'}^{\mu'}$) giving the Ekedahl-Oort strata on $\ES_{0,K}(G,X)$ (resp. $\ES_{0,K'}(G',X')$).

The question is, whether there is any compatibility between $f$, $\zeta$ and $\zeta'$. We have the following result.
\begin{theorem}
Assume that $f_{\mathbb{Q}_p}$ extends to a morphism of reductive group schemes over $\mathbb{Z}_p$, then there is a canonical commutative diagram $$\xymatrix{
\ES_{0,K}(G,X)\ar[d]^{\zeta}\ar[r]^{f}& \ES_{0,K'}(G',X')_\kappa\ar[d]^{\zeta'\otimes\kappa} \\
G_0\mathrm{-zip}^\mu_\kappa\ar[r]^{\alpha} & G'_0\mathrm{-zip}^{\mu'}_{\kappa'}\otimes \kappa.
}$$
\end{theorem}
The proof to the above result will be given in the second section. 

As we have seen, the second question is not yet totally solved. More questions will be raised and studied in the author's future research.

\section[Independence of symplectic embeddings]{Independence of symplectic embeddings}

Let $(G,X)$ be a Shimura datum of Hodge type with good reduction at a prime $p>2$. Let $E$ be the reflex field and $v$ be a place of $E$ over $p$. The residue field at $v$ will be denoted by $\kappa$. Let $K_p\subseteq G(\mathbb{Q}_p)$ be a hyperspecial subgroup, and $K^p\subseteq G(\mathbb{A}_f^p)$ be a compact open subgroup which is small enough. Let $K$ be $K_p\times K^p$. Then by \cite{CIMK}, the Shimura varieties $\Sh_K(G,X)$ has an integral canonical model $\ES_K(G,X)$ which is smooth over~$O_{E,(v)}$.

Let $\ES_0$ be the special fiber of $\ES_K(G,X)$. By the main results of \cite{EOZ}, there is a theory of Ekedahl-Oort stratification on $\ES_0$. To define the stratification, we need to fix a symplectic embedding, while the variety $\ES_0$ is independent of symplectic embeddings. A natural question is whether different symplectic embeddings give the same stratification.

The above question is not yet precise enough to work with. Let us first recall how Ekedahl-Oort stratifications are defined and raise precise questions.

\subsection[Integral canonical models]{Integral canonical models}\label{intcanmod}

Let $K_p$ and $G$ be as above, then by \cite{INTV} Proposition 3.1.2.1 c) and e), $G_{\mathbb{Q}_p}$ extends uniquely to a reductive group $\widehat{\mathcal {G}}/\mathbb{Z}_{p}$ such that $K_p=\widehat{\mathcal {G}}(\mathbb{Z}_{p})$. More precisely, for any two extensions $e_1:G_{\mathbb{Q}_p}\rightarrow \widehat{\mathcal {G}}_1$ and $e_2:G_{\mathbb{Q}_p}\rightarrow \widehat{\mathcal {G}}_2$ such that $e_1(K_p)=\widehat{\mathcal {G}}_1(\mathbb{Z}_p)$ and $e_2(K_p)=\widehat{\mathcal {G}}_2(\mathbb{Z}_p)$, there is an unique isomorphism $f:\widehat{\mathcal {G}}_1\rightarrow \widehat{\mathcal {G}}_2$ such that $f\circ e_1=e_2$ and that $f(\widehat{\mathcal {G}}_1(\mathbb{Z}_p))=\widehat{\mathcal {G}}_2(\mathbb{Z}_p)$.

Let $i:(G,X)\hookrightarrow (\mathrm{GSp}(V,\psi),X')$ be a symplectic embedding. Then by \cite{CIMK} Lemma 2.3.1, there exists a $\mathbb{Z}_p$-lattice $L_{\mathbb{Z}_p}\subseteq V_{\mathbb{Q}_p}$, such that $i_{\mathbb{Q}_p}:G_{\mathbb{Q}_p}\subseteq\mathrm{GL}(V_{\mathbb{Q}_p})$ extends uniquely to a closed embedding $\widehat{\mathcal {G}}\hookrightarrow \mathrm{GL}(L_{\mathbb{Z}_p}).$ So there is a $\mathbb{Z}$-lattice $L\subseteq V$ such that $\mathcal {G}$, the Zariski closure of $G$ in $\mathrm{GL}(L_{\mathbb{Z}_{(p)}})$, is reductive, as the base change to $\mathbb{Z}_p$ of $\mathcal {G}$ is $\widehat{\mathcal {G}}$. Moreover, we can assume $L$ is such that $L^\vee\supseteq L$. Let $d=|L^\vee/L|$ and $g=\frac{1}{2}\mathrm{dim}(V)$, then the integral canonical model $\ES_K(G,X)$ of $\mathrm{Sh}_K(G,X)$ is constructed as follows. We can choose $K'\subseteq\mathrm{GSp}(V,\psi)(\mathbb{A}_f)$ small enough such that $K'\supseteq K$ and that $\mathrm{Sh}_{K'}(\mathrm{GSp}(V,\psi),X)$ affords a moduli interpretation. There is a finite morphism $f:\mathrm{Sh}_K(G,X)\rightarrow \mathrm{Sh}_{K'}(\mathrm{GSp}(V,\psi),X)_E$. Let $\mathscr{A}_{g,d,K'/\mathbb{Z}_{(p)}}$ be the moduli scheme of abelian schemes over $\mathbb{Z}_{(p)}$-schemes with a polarization of degree $d$ and level $K'$ structure. Then $\mathscr{A}_{g,d,K'/\mathbb{Z}_{(p)}}\otimes \mathbb{Q}=\mathrm{Sh}_{K'}(\mathrm{GSp}(V,\psi),X)$. The integral canonical model $\ES_K(G,X)$ is the normalization of the Zariski closure of $\mathrm{Sh}_K(G,X)$ in $\mathscr{A}_{g,d,K'/\mathbb{Z}_{(p)}}\otimes O_{E,(v)}$. Here the word ``normalization'' make sense. As $\Sh_K(G,X)$ is regular, and on each open affine, $O_{\ES_K(G,X)}$ is obtain by taking elements in $O_{\Sh_K(G,X)}$ that is integral over $O_{\mathscr{A}_{g,d,K'/O_{E,(v)}}}$.

Note that we didn't assume that $K'$ is such that the morphism $f$ is a closed embedding. Because if we take $K''\subseteq K'$ small enough such that the induced morphism $$g:\mathrm{Sh}_K(G,X)\rightarrow \mathrm{Sh}_{K''}(\mathrm{GSp}(V,\psi),X)_E$$ is a closed embedding, then $f$ factors through $g$. The natural morphism $\mathscr{A}_{g,d,K''/\mathbb{Z}_{(p)}}\rightarrow \mathscr{A}_{g,d,K'/\mathbb{Z}_{(p)}}$ is finite, so the normalization gives the same $\ES_K(G,X)$. The special fiber of $\ES_K(G,X)$ will be denoted by $\ES_0$.

\subsection[$G$-zips]{$G$-zips}\label{G-zips}

Let $G_0$ (resp. $L_0$) be the special fiber of $\mathcal {G}$ (resp. $L_{\mathbb{Z}_p}$). We remark that $G_0$ is uniquely determined by $(G,K_p)$, as it is also the special fiber $\widehat{\mathcal {G}}$ which is uniquely determined by $(G,K_p)$. But $L_0$ is not unique, there might be many choices. By \cite{EOZ}, the Shimura datum $(G,X)$ determines a cocharacter $\mu:\mathbb{G}_{m,W(\kappa)}\rightarrow \widehat{\mathcal {G}}_{W(\kappa)}$ which is unique up to $\widehat{\mathcal {G}}(W(\kappa))$-conjugacy. The special fiber of $\mu$ will still be denoted by $\mu$.

\begin{set}\label{set G-zip}
We start with $G_0$ and $\mu:\mathbb{G}_{m,\kappa}\rightarrow G_{0,\kappa}$. For an $\mathbb{F}_p$-scheme $S$, let $\sigma:S\rightarrow S$ be the absolute Frobenius. For an $S$-scheme $T$, we will write $T^{(p)}$ for the pull back of $T$ via $\sigma$. In particular, we will write $\mu^{(p)}$ for the pull back via Frobenius of $\mu$. Note that it is a cocharacter of $G_{0,\kappa}$.

Let $P_{+}$ (resp. $P_{-}$) be the unique parabolic subgroup of
$G_{0,\kappa}$ such that its Lie algebra is the sum of spaces with
non-negative weights (resp. non-positive weights) in
$\text{Lie}(G_{0,\kappa})$ under $\text{Ad}\circ\mu$.

Let $U_{+}$ (resp. $U_{-}$) be the unipotent radical of $P_{+}$
(resp. $P_{-}$), and $L$ be the common Levi subgroup of $P_{+}$ and
$P_{-}$. Note that $L$ is also the centralizer of
$\mu$.
\end{set}

\begin{definition}\label{G-ziptypekappa}
Let $S$ be a scheme over $\kappa$. A $G_0$-zip of type
$\mu$ over $S$ is a tuple $\underline{I}=(I, I_+, I_-, \iota)$
consisting of a right $G_\kappa$-torsor $I$ over~$S$, a right
$P_+$-torsor $I_+\subseteq I$, a right $P_-^{(p)}$-torsor
$I_-\subseteq I$, and an isomorphism of $L^{(p)}$-torsors
$\iota:I^{(p)}_+/U^{(p)}_+\rightarrow I_-/U^{(p)}_-$.

A morphism $(I, I_+, I_-, \iota)\rightarrow (I', I'_+, I'_-,
\iota')$ of $G_0$-zips of type $\mu$ over $S$ consists of equivariant
morphisms $I\rightarrow I'$ and $I_{\pm}\rightarrow I'_{\pm}$ that
are compatible with inclusions and the isomorphisms $\iota$
and~$\iota'$.
\end{definition}

The category of $G_0$-zips of type $\mu$ over a $\kappa$-scheme $S$
will be denoted by $G_0\texttt{-Zip}_\kappa^{\mu}(S)$. They form a
fibered category $G_0\texttt{-Zip}_\kappa^{\mu}$ over the category of
$\kappa$-schemes if we only consider isomorphisms as morphisms.

Pink, Wedhorn and Ziegler proved the following result.
\begin{theorem}
The fibered category $G_0\texttt{-Zip}_\kappa^{\mu}$ is a smooth
algebraic stack over $\kappa$ of dimension 0.
\end{theorem}
\begin{proof}
This is \cite{zipaddi} Corollary 3.12.
\end{proof}

\subsection[Ekedahl-Oort strata]{Ekedahl-Oort strata}\label{EOstrata}

Now we will explain how to construct Ekedahl-Oort stratification follow \cite{EOZ}. Note that we will NOT follow \cite{EOZ} strictly, as it seems more natural to compare $L^\vee$ with cohomologies, see also \cite{LRKisin} and \cite{muordinary}. Our theory of Ekedahl-Oort stratification is base on the theory of $G_0$-zips of type $\mu$ defined and studied by Pink, Wedhorn and Ziegler in \cite{zipaddi}.

Let $\mathcal {A}$ be the pull back to $\ES_K(G,X)$ of the universal abelian scheme on $\mathscr{A}_{g,d,K'/\mathbb{Z}_{(p)}}$, and $\mathcal {V}$ be $\Hdr(\mathcal {A}/\ES_K(G,X))$. Let $L\subseteq V$ and $\mathcal {G}$ be as in \ref{intcanmod}. Then by \cite{CIMK} Proposition 1.3.2, there is a tensor $s\in L_{\mathbb{Z}_{(p)}}^\otimes$ defining $\mathcal {G}\subseteq \mathrm{GL}(L_{\mathbb{Z}_{(p)}}).$ Corollary 2.3.9 of \cite{CIMK} implies that the tensor $s\in L_{\mathbb{Z}_{(p)}}^\otimes$ induces a section $\sdr\in \mathcal {V}^\otimes$. By \cite{EOZ} Lemma 2.3.2 1), the scheme $$I=\Isom_{\ES_K(G,X)}\big((L^\vee_{\mathbb{Z}_{(p)}},s)\otimes O_{\ES_K(G,X)},(\mathcal {V}, \sdr)\big)$$ is a right $\mathcal {G}$-torsor.

The first main result of \cite{EOZ} is as follows.
\begin{set}\label{F-zip attached to abs}
Still write $\V$, $s$, $\sdr$ and $I$ the reduction mod $p$ of $\V$, $s$, $\sdr$ and $I$. Let $F:\V^{(p)}\rightarrow \V$ and $V:\V\rightarrow \V^{(p)}$ be the Frobenius and Verschiebung on $\V$ respectively. Let $\delta:\V\rightarrow \V^{(p)}$ be the semi-linear map sending $v$ to $v\otimes 1$. Then we have a semi-linear map $F\circ \delta:\V\rightarrow \V$. There is a descending filtration $\V\supseteq \mathrm{ker}(F\circ \delta)\supseteq 0$ and an ascending filtration $0\subseteq \mathrm{im}(F)\subseteq \V.$ The morphism $V$ induces an isomorphism $\V/\mathrm{im}(F)\rightarrow \mathrm{ker}(F)$ whose inverse will be denoted by $V^{-1}$. Then $F$ and $V^{-1}$ induce isomorphisms $$\varphi_0:(\V/\mathrm{ker}(F\circ \delta))^{(p)}\rightarrow \mathrm{ker}(F)$$
and $$\varphi_1:(\mathrm{ker}(F\circ \delta))^{(p)}\rightarrow \V/(\mathrm{im}(F)).$$
\end{set}

\begin{set}\label{phi in standard G-zip}
Let $L_0$ be the special fiber of $L_{\mathbb{Z}_{(p)}}$. The cocharacter $$\mu:\mathbb{G}_{m,\kappa}\rightarrow G_{0,\kappa}\subseteq \mathrm{GL}(L_{0,\kappa})\cong\mathrm{GL}(L^\vee_{0,\kappa})$$ induces an $F$-zip structure on $L^\vee_{0,\kappa}$ as follows. Let $(L^\vee_{0,\kappa})^0$ (resp. $(L^\vee_{0,\kappa})^1$) be the subspace of $L^\vee_{0,\kappa}$ of weight $0$ (resp. 1) with respect to $\mu$, and $(L^\vee_{0,\kappa})_0$ (resp. $(L^\vee_{0,\kappa})_1$) be the subspace of $L^\vee_{0,\kappa}$ of weight $0$ (resp. 1) with respect to $\mu^{(p)}$. Then we have a descending filtration
$L^\vee_{0,\kappa}\supseteq (L^\vee_{0,\kappa})^1\supseteq 0$ and an ascending filtration $0\subseteq (L^\vee_{0,\kappa})_0\subseteq L^\vee_{0,\kappa}$. Let $\xi:L^\vee_{0,\kappa}\rightarrow (L^\vee_{0,\kappa})^{(p)}$ be the isomorphism given by $l\otimes k\mapsto l\otimes 1\otimes k$, $\forall\ l\in L_0^\vee$, $\forall\ k\in \kappa$. Then $\xi$ induces isomorphisms

$$\phi_0:(L^\vee_{0,\kappa})^{(p)}/((L^\vee_{0,\kappa})^1)^{(p)}\stackrel{\mathrm{pr}_2}{\rightarrow}((L^\vee_{0,\kappa})^0)^{(p)} \stackrel{\xi^{-1}}{\longrightarrow}(L^\vee_{0,\kappa})_0$$ and
$$\phi_1:((L^\vee_{0,\kappa})^1)^{(p)}\stackrel{\xi^{-1}}{\longrightarrow}((L^\vee_{0,\kappa})_1\simeq L^\vee_{0,\kappa}/(L^\vee_{0,\kappa})_0.$$
\end{set}

\begin{theorem}\label{G-zipES_0}

\

1) Let $I_+\subseteq I$ be the closed subscheme
$$I_+:=\Isom_{\ES_0}\big((L_{0,\kappa}^\vee, s, (L^\vee_{0,\kappa})^1)\otimes O_{\ES_0},\ (\V,
\sdr, \mathrm{ker}(F\circ \delta))\big).$$Then $I_+$ is a $P_+$-torsor over $\ES_0$.

2) Let $I_-\subseteq I$ be the closed subscheme
$$I_-:=\Isom_{\ES_0}\big((L_{0,\kappa}^\vee, s, (L^\vee_{0,\kappa})_0)\otimes O_{\ES_0},\ (\V,
\sdr, \mathrm{im}(F))\big).$$Then $I_-$ is a $P_-^{(p)}$-torsor over
$\ES_0$.

3) Let $\iota:I_+^{(p)}/U_+^{(p)}\rightarrow I_-/U_-^{(p)}$ be the morphism
induced by
\begin{equation*}
\begin{split}
 I_+^{(p)}&\rightarrow I_-/U_-^{(p)}\\
f&\mapsto (\varphi_0\oplus \varphi_1)\circ\mathrm{gr}(f)\circ(\phi_0^{-1}\oplus \phi_1^{-1}), \forall\ S/\ES_0\text{ and }\forall\ f\in I_+^{(p)}(S).
 \end{split}
 \end{equation*}
Then $\iota$ is an isomorphism of $L^{(p)}$-torsors.

Hence the tuple $(I,I_+,I_-,\iota)$ is a $G_0$-zip of type $\mu$ over $\ES_0$.
\end{theorem}
\begin{proof}
This is \cite{EOZ} Theorem 2.4.1.
\end{proof}
The $G_0$-zip of type $\mu$ over $\ES_0$ constructed above induces a morphism $\zeta:\ES_0\rightarrow G_0\texttt{-Zip}_\kappa^{\mu}$. As we have seen, to construct $\zeta$, we need to choose a $\mathbb{Z}_{(p)}$-model $\mathcal {G}$ of $G$, a symplectic embedding $i:(G,X)\hookrightarrow (\mathrm{GSp}(V,\psi),X')$, a $\mathbb{Z}_{(p)}$-lattice $L_{\mathbb{Z}_{(p)}}\subseteq V$, and a tensor $s\in L_{\mathbb{Z}_{(p)}}$ defining $\mathcal {G}$. So, by independence of symplectic embeddings, we mean that $\zeta$ is independent of the choices of $\mathcal {G}$, $i$, $L_{\mathbb{Z}_{(p)}}$ and $s$.

\subsection[Uniqueness of $\mathcal {G}$]{Uniqueness of $\mathcal {G}$}

The $\mathbb{Z}_{(p)}$-model $\mathcal {G}$ of $G$ we obtained is actually unique.
\begin{proposition}
Let $V_1$ and $V_2$ be two finite dimensional $\mathbb{Q}$-vector spaces. Let $i_1:G\rightarrow\mathrm{GL}(V_1)$ and $i_2:G\rightarrow\mathrm{GL}(V_2)$ be two closed embeddings of reductive groups. Assume that there is a $\mathbb{Z}_{(p)}$-lattice $L_1\subseteq V_1$ (resp. $L_2\subseteq V_2$) such that the Zariski closure $\mathcal {G}_1$ (resp. $\mathcal {G}_2$) of $G$ in $\mathrm{GL}(L_1)$ (resp. $\mathrm{GL}(L_2)$) is reductive. Then $\mathcal {G}_1\cong\mathcal {G}_2.$
\end{proposition}
\begin{proof}
Let $V=V_1\oplus V_2$ be the direct sum of the two representations $i_1$ and $i_2$, let $L=L_1\oplus L_2$. The we have a sequence of closed embeddings $\mathcal {G}_1\times \mathcal {G}_2\subseteq\mathrm{GL}(L_{1})\times \mathrm{GL}(L_{2})\subseteq \mathrm{GL}(L).$ Let $\widehat{\mathcal {G}}/\mathbb{Z}_p$ be the reductive model of $G_{\mathbb{Q}_p}$ and $\mathcal {G}_3$ be the Zariski closure of $G$ in $\mathrm{GL}(L)$. Then $\mathcal {G}_{1,\mathbb{Z}_p}=\widehat{\mathcal {G}}= \mathcal {G}_{2,\mathbb{Z}_p}$. Flat base-change implies that $\mathcal {G}_{3,\mathbb{Z}_p}$ is the diagonal subgroup of $\widehat{\mathcal {G}}\times \widehat{\mathcal {G}}=\mathcal {G}_{1,\mathbb{Z}_p}\times \mathcal {G}_{2,\mathbb{Z}_p}$. So $\mathcal {G}_3$ is reductive. Note that the morphism $\mathcal {G}_3\rightarrow \mathcal {G}_1\times \mathcal {G}_2\stackrel{p_1}{\longrightarrow}\mathcal {G}_1$ is an isomorphism, so $\mathcal {G}_3\cong \mathcal {G}_1$. Similarly, $\mathcal {G}_3\cong \mathcal {G}_2$.
\end{proof}

Let $\mathcal {G}/\mathbb{Z}_{(p)}$ be a reductive model of $G$. Then there exists a free $\mathbb{Z}_{(p)}$-module $M$ of finite rank such that there is a closed embedding $\mathcal {G}\hookrightarrow \mathrm{GL}(M)$. The generic fiber of this embedding satisfies the condition of the above proposition. So two reductive models over $\mathbb{Z}_{(p)}$ of $G$ must be isomorphic.

\subsection[Comparing $G_0$-zips (I)]{Comparing $G_0$-zips (I)}

We will first show that the morphism $\zeta$ does not depend on the choices of $s$, once $\mathcal {G}$, $i$ and $L_{\mathbb{Z}_{(p)}}$ are fixed. Let us recall our notations and constructions in \ref{intcanmod}.

For the symplectic embedding $i:(G,X)\subseteq (\mathrm{GSp}(V,\psi),X')$ and a the chosen reductive model $\mathcal {G}/\mathbb{Z}_{(p)}$ of $G$, there is a $\mathbb{Z}$-lattice $L\subseteq V$ such that the Zariski closure of $G$ in $\mathrm{GL}(L_{\mathbb{Z}_{(p)}})$ is $\mathcal {G}$ and that $\mathcal {G}$ is defined by a tensor $s\in L_{\mathbb{Z}_{(p)}}^\otimes$. One can choose $L$ such that $L^\vee\supseteq L$. Let $d=|L^\vee/L|$, $g=\frac{1}{2}\mathrm{dim}(V)$, $K_p=\mathcal {G}(\mathbb{Z}_p)$ and $K=K_pK^p$ with $K^p\subseteq G(\mathbb{A}_f^p)$ small enough. Then the integral canonical model $\ES_K(G,X)$ of $\mathrm{Sh}_K(G,X)$ is constructed as follows. We can choose $K'\subseteq\mathrm{GSp}(V,\psi)(\mathbb{A}_f)$ small enough such that $K'\supseteq K$ and that $\mathrm{Sh}_{K'}(\mathrm{GSp}(V,\psi),X)$ affords a moduli interpretation. There is a finite morphism $f:\mathrm{Sh}_K(G,X)\rightarrow \mathrm{Sh}_{K'}(\mathrm{GSp}(V,\psi),X)_E$. Let $\mathscr{A}_{g,d,K'/\mathbb{Z}_{(p)}}$ be the moduli scheme of abelian schemes over $\mathbb{Z}_{(p)}$-schemes with a polarization of degree $d$ and level $K'$ structure. Then $\mathscr{A}_{g,d,K'/\mathbb{Z}_{(p)}}\otimes \mathbb{Q}=\mathrm{Sh}_{K'}(\mathrm{GSp}(V,\psi),X)$, and the integral canonical model $\ES_K(G,X)$ is the normalization of the Zariski closure of $\mathrm{Sh}_K(G,X)$ in $\mathscr{A}_{g,d,K'/\mathbb{Z}_{(p)}}\otimes O_{E,(v)}$.

Let $\mathcal {A}$ be the pull back to $\ES_K(G,X)$ of the universal abelian scheme on $\mathscr{A}_{g,d,K'/\mathbb{Z}_{(p)}}$, and $\mathcal {V}$ be $\Hdr(\mathcal {A}/\ES_K(G,X))$. Then the tensor $s\in L_{\mathbb{Z}_{(p)}}^\otimes$ induces a section $\sdr\in \mathcal {V}^\otimes$. For a different choice of $s'\in L_{\mathbb{Z}_{(p)}}^\otimes$, we have another section $\sdr'\in \mathcal {V}^\otimes.$
We have two $\mathcal {G}$-torsors
\begin{equation*}
\begin{split}
&I=\Isom_{\ES_K(G,X)}\big((L^\vee_{\mathbb{Z}_{(p)}},s)\otimes O_{\ES_K(G,X)},(\mathcal {V}, \sdr)\big)\\
\text{and }\ \ &I=\Isom_{\ES_K(G,X)}\big((L^\vee_{\mathbb{Z}_{(p)}},s')\otimes O_{\ES_K(G,X)},(\mathcal {V}, \sdr')\big).
 \end{split}
 \end{equation*}
\begin{lemma}\label{torI}
The two $\mathcal {G}$-torsors $I$ and $I'$ are canonically isomorphic.
\end{lemma}
\begin{proof}
We will show that $I$ and $I'$ are the same closed subscheme of $\Isom_{\ES_K(G,X)}(L^\vee\otimes O_{\ES_K(G,X)},\mathcal {V})$. Let $$I'':=\Isom_{\ES_K(G,X)}\big((L^\vee_{\mathbb{Z}_{(p)}},s,s')\otimes O_{\ES_K(G,X)},(\mathcal {V}, \sdr, s'_{\dr})\big).$$
Then it is a closed subscheme of both $I$ and $I'$. But $I''$ is also a $\mathcal {G}$-torsor over $\ES_K(G,X)$, so $I=I''=I'$.
\end{proof}
Let $\ES_{0,K}(G,X)$ (resp. $G_0$) be the special fiber of $\ES_K(G,X)$ (resp. $\mathcal {G}$). The construction in \ref{EOstrata}, especially Theorem \ref{G-zipES_0}, gives a $G_0$-zip of type $\mu$ $(I,I_+,I_-,\iota)$ on $\ES_{0,K}(G,X)$, using $L^\vee_{\mathbb{Z}_{(p)}},s,\mathcal {V},\sdr$ and the $F$-zip structure on $\mathcal {V}$. Similarly, there is a $G_0$-zip $(I',I'_+,I'_-,\iota')$ attached to $L^\vee_{\mathbb{Z}_{(p)}},s',\mathcal {V},s'_{\dr}$.
\begin{corollary}\label{inpe of tensors}
The two $G_0$-zips of type $\mu$ $(I,I_+,I_-,\iota)$ and $(I',I'_+,I'_-,\iota')$ on $\ES_{0,K}(G,X)$ are canonically isomorphic.
\end{corollary}
\begin{proof}
By Lemma \ref{torI}, the torsors $I$ and $I'$ are canonically isomorphic. Noting that $(I_+,I_-,\iota)$ and $(I'_+,I'_-,\iota')$ are constructed using Frobenius and Verschiebung on $\V$, the two $G_0$-zips are canonically isomorphic.
\end{proof}

\subsection[Symplectic embeddings]{Symplectic embeddings}\label{sum of ebd}

Let $i_1:(G,X)\hookrightarrow (\mathrm{GSp}(V_1,\psi_1),X_1)$ and $i_2:(G,X)\hookrightarrow (\mathrm{GSp}(V_2,\psi_2),X_2)$ be two symplectic embeddings. We can construct another symplectic embedding as follows.

By definition of symplectic similitude groups, there is a character $\chi_1:\mathrm{GSp}(V_1,\psi_1)\rightarrow \mathbb{G}_m$, such that $\mathrm{GSp}(V_1,\psi_1)$ acts on $\psi_1$ via $\chi_1$. Note that changing $\chi_1$ to a power of it will not change the symplectic similitude group. Similarly, we have $\chi_2:\mathrm{GSp}(V_2,\psi_2)\rightarrow \mathbb{G}_m$. Let $w:\mathbb{G}_m\rightarrow G$ be the weight cocharacter of $G$. Then $\chi_1\circ w$ and $\chi_2\circ w$ are two characters $\mathbb{G}_m\rightarrow \mathbb{G}_m$ of weights $m_1$ and $m_2$ respectively. After changing $\chi_1$ to $\chi_1^{m_2}$ and $\chi_2$ to $\chi_2^{m_1}$, we see that $G$ acts on $\psi_1$ and $\psi_2$ via the same character.

Let $V=V_1\oplus V_2$ and $\psi: V\times V\rightarrow \mathbb{Q}$ be such that
$$\psi\big((v_1,v_2),(v_1',v_2')\big)=\psi_1(v_1,v_1')+\psi_2(v_2,v_2'),\ \ \ \forall\ v_1,v_1'\in V_1 \ \mathrm{and}\ \forall\ v_2,v_2'\in V_2.$$ Then $G\subseteq \mathrm{GSp}(V,\psi)$, and this embedding induces an embedding of Shimura data $$(G,X)\subseteq (\mathrm{GSp}(V,\psi),X').$$

\subsection[Comparing $G_0$-zips (II)]{Comparing $G_0$-zips (II)}\label{Comparing G_0-zips (II)}

Now we will show that the morphism $\zeta:\ES_{0,K}(G,X)\rightarrow G_0\mathrm{-zip}^\mu_\kappa$ is independent of choices of symplectic embeddings, reductive models and lattices. Note that $\zeta$ is independent of $K$. More precisely, for $K\subseteq K'$, there is a commutative diagram
$$\xymatrix{
  \ES_{0,K}(G,X)\ar[dr]_{\zeta} \ar[r]
                & \ES_{0,K'}(G,X) \ar[d]^{\zeta'}  \\
                &G_0\mathrm{-zip}^\mu_\kappa},$$
inducing a $G(\mathbb{A}_f^p)$-equivariant morphism $$\ES_{0,K_p}(G,X)=\varprojlim_{K^p}\ES_{K_pK^p}(G,X)\longrightarrow G_0\mathrm{-zip}^\mu_\kappa.$$ Here the $G(\mathbb{A}_f^p)$-action on $G_0\mathrm{-zip}^\mu_\kappa$ is the trivial action, and that on $\ES_{0,K_p}(G,X)$ is the unique one induced by the action on $\Sh_{K_p}(G,X)$. So, we can shrink $K^p$ if necessary.

Let $i_1:(G,X)\hookrightarrow (\mathrm{GSp}(V_1,\psi_1),X_1)$ and $i_2:(G,X)\hookrightarrow (\mathrm{GSp}(V_2,\psi_2),X_2)$ be two symplectic embeddings. Let $\mathcal {G}$ be the reductive model of $G$ over $\mathbb{Z}_{(p)}$ such that $\mathcal {G}(\mathbb{Z}_p)=K_p.$ 

There are lattices $L_t\subseteq V_t$, $t=1,2$, such that\\
1. $\psi_t$ takes integral value on $L_t$. \\
2. The Zariski closure of $G$ in $\mathrm{GL}(L_{t,\mathbb{Z}_{(p)}})$ is $\mathcal {G}$.

Let $d_t=|L_t^\vee/L_t|$, $g_t=\frac{1}{2}\mathrm{dim}(V_t)$, and $n\geq 3$ be an integer such that $(n,p)=1$. Let $\mathscr{A}_{g_t,d_t,n/\mathbb{Z}_{(p)}}$ be the moduli scheme of abelian schemes over $\mathbb{Z}_{(p)}$-schemes of relative dimension $g_t$ with a polarization $\lambda_t$ of degree $d_t$ and a level $n$ structure $\tau_t$. We write $(\mathcal {A}_t,\lambda_t,\tau_t)$ for the universal family on $\mathscr{A}_{g_t,d_t,n/\mathbb{Z}_{(p)}}$. Let $K^p\subseteq G(\mathbb{A}_f^p)$ be small enough such that there are natural morphisms $\Sh_K(G,X)\rightarrow \mathscr{A}_{g_1,d_1,n/E}$ and $\Sh_K(G,X)\rightarrow \mathscr{A}_{g_2,d_2,n/E}$. Then by the construction of the integral canonical model, there are natural finite morphisms $i_1:\ES_K(G,X)\rightarrow \mathscr{A}_{g_1,d_1,n/O_{E,(v)}}$ and $i_2:\ES_K(G,X)\rightarrow \mathscr{A}_{g_2,d_2,n/O_{E,(v)}}$.

Let $V$ be $V_1\oplus V_2$, $L$ be $L_1\oplus L_2$ and $\psi:V\times V\rightarrow \mathbb{Q}$ be such that
$$\psi\big((v_1,v_2),(v_1',v_2')\big)=\psi_1(v_1,v_1')+\psi_2(v_2,v_2'),\ \ \ \forall\ v_1,v_1'\in V_1 \ \mathrm{and}\ \forall\ v_2,v_2'\in V_2.$$
Then by \ref{sum of ebd}, there is an embedding of Shimura data $i:(G,X)\rightarrow (\mathrm{GSp}(V,\psi),Y)$. Consider the diagram
$$\xymatrix{
                & \mathscr{A}_{g_1,d_1,n/\mathbb{Z}_{(p)}}\times \mathscr{A}_{g_2,d_2,n/\mathbb{Z}_{(p)}} \ar[dl]^{p_1}\ar[dr]_{p_2}\\
\mathscr{A}_{g_1,d_1,n/\mathbb{Z}_{(p)}} & &\mathscr{A}_{g_2,d_2,n/\mathbb{Z}_{(p)}}}.$$
The abelian scheme $p_1^*(\mathcal {A}_1,\lambda_1,\tau_1)\times p_2^*(\mathcal {A}_2,\lambda_2,\tau_2)$ on $\mathscr{A}_{g_1,d_1,n/\mathbb{Z}_{(p)}}\times \mathscr{A}_{g_2,d_2,n/\mathbb{Z}_{(p)}}$ is an abelian scheme of dimension $g_1+g_2$ with a polarization of degree $d_1d_2$ and level $n$ structure. There is a unique morphism $$i':\mathscr{A}_{g_1,d_1,n/\mathbb{Z}_{(p)}}\times \mathscr{A}_{g_2,d_2,n/\mathbb{Z}_{(p)}}\longrightarrow \mathscr{A}_{g_1+g_2,d_1d_2,n/\mathbb{Z}_{(p)}}$$ such that $$i'^*(\mathcal {A},\lambda,\tau)=p_1^*(\mathcal {A}_1,\lambda_1,\tau_1)\times p_2^*(\mathcal {A}_2,\lambda_2,\tau_2),$$ where $(\mathcal {A},\lambda,\tau)$ is the universal family on $\mathscr{A}_{g_1+g_2,d_1d_2,n/\mathbb{Z}_{(p)}}$.

By the construction of $\ES_K(G,X)$, we have a commutative diagram
$$\xymatrix{
\ES_K(G,X)\ar[rr]^(0.4){i_1\times i_2}\ar[dr]^{i_1}\ar[ddr]^{i_2}&& \mathscr{A}_{g_1,d_1,n/O_{E,(v)}}\times \mathscr{A}_{g_2,d_2,n/O_{E,(v)}}\ar[r]^(0.6){i'} \ar[dl]_{p_1}\ar[ddl]_{p_2}&\mathscr{A}_{g_1+g_2,d_1d_2,n/O_{E,(v)}}\\
& \mathscr{A}_{g_1,d_1,n/O_{E,(v)}}&\\
& \mathscr{A}_{g_2,d_2,n/O_{E,(v)}}&
}$$
such that the generic fiber of $i'\circ (i_1\times i_2)$ is induced by $i$. We will write $i$ for $i'\circ (i_1\times i_2)$. The pull back via $i$ of the universal family on $\mathscr{A}_{g_1+g_2,d_1d_2,n/O_{E,(v)}}$ is precisely $i_1^*(\mathcal {A}_1,\lambda_1,\tau_1)\times i_2^*(\mathcal {A}_2,\lambda_2,\tau_2).$

Let $\mathcal {A}_1,\mathcal {A}_2,\mathcal {A}$ be the pull back to $\ES_{K}(G,X)$ of the universal abelian schemes on $\mathscr{A}_{g_1,d_1,n}$, $\mathscr{A}_{g_2,d_2,n}$ and $\mathscr{A}_{g_1+g_2,d_1d_2,n}$ respectively. Then $\mathcal {A}=\mathcal {A}_1\times\mathcal {A}_2.$ Let $\V_t=\Hdr(\mathcal {A}_t/\ES_K(G,X))$, $t=1,2$, and $\V=\Hdr(\mathcal {A}/\ES_K(G,X))$. Then $\V=\V_1\oplus \V_2$.

Let $L=L_1\oplus L_2$ and $V=V_1\oplus V_2$ be as before. Let $s_t\in L_{t,\mathbb{Z}_{(p)}}^\otimes$ be a tensor defining $\mathcal {G}\subseteq \mathrm{GL}(L_{t,\mathbb{Z}_{(p)}})$, for $t=1,2$. By our construction, we have a sequence of closed embeddings
\begin{subeqn}\label{seq of closed sbgp}
\mathcal {G}\subseteq\mathcal {G}\times \mathcal {G}\subseteq\mathrm{GL}(L_{1,\mathbb{Z}_{(p)}})\times \mathrm{GL}(L_{2,\mathbb{Z}_{(p)}})\subseteq \mathrm{GL}(L_{\mathbb{Z}_{(p)}}).
\end{subeqn}
Here the first embedding is the diagonal embedding.

By \cite{CIMK} Proposition 1.3.2, $\mathcal {G}\subseteq \mathrm{GL}(L_{\mathbb{Z}_{(p)}})$ is defined by a tensor $s\in L_{\mathbb{Z}_{(p)}}^\otimes.$ We need some explicit conditions that cut out $\mathcal {G}\times \mathcal {G}\subseteq \mathrm{GL}(L_{\mathbb{Z}_{(p)}})$. First note that $\mathrm{GL}(L_{1,\mathbb{Z}_{(p)}})\times \mathrm{GL}(L_{2,\mathbb{Z}_{(p)}})$ is the subgroup of $\mathrm{GL}(L_{\mathbb{Z}_{(p)}})$ respecting the splitting $L_{\mathbb{Z}_{(p)}}=L_{1,\mathbb{Z}_{(p)}}\oplus L_{2,\mathbb{Z}_{(p)}}.$ So the group scheme $\mathcal {G}\times \mathcal {G}$ is such that for any $\mathbb{Z}_{(p)}$-algebra $R$,
$$\mathcal {G}\times \mathcal {G}(R)=\{g\in \mathrm{GL}(L_{\mathbb{Z}_{(p)}})(R)\mid g(L_t\otimes R)=L_t\otimes R\text{ and }g(s_t\otimes 1)=s_t\otimes 1 \text{ for }t=1,2\}.$$But then $\mathcal {G}$ will be the group scheme such that for any $\mathbb{Z}_{(p)}$-algebra $R$,
$$\mathcal {G}(R)=\{g\in \mathrm{GL}(L_{\mathbb{Z}_{(p)}})(R)\mid g(L_t\otimes R)=L_t\otimes R\text{ and }g(s_t\otimes 1)=s_t\otimes 1 \text{ for }t=1,2, g(s\otimes1)=s\otimes 1\}.$$ Clearly, if we remove the conditions on $L_2$ and $s_2$, we get the same group scheme.

Let $s_{1,\dr}\in \V_1^\otimes$ (resp. $s_{\dr}\in \V^\otimes$) be the section corresponding to $s_1$ (resp. $s$). Let
$$I=\Isom_{\ES_K(G,X)}\big((L^\vee_{\mathbb{Z}_{(p)}},L^\vee_{1,\mathbb{Z}_{(p)}},s_1, s)\otimes O_{\ES_K(G,X)},\ (\mathcal {V}, \V_1, s_{1,\dr},\sdr)\big).$$ where $L^\vee_{1,\mathbb{Z}_{(p)}}\subseteq L^\vee_{\mathbb{Z}_{(p)}}$ is given by taking dual of the surjection $p_1:L_{\mathbb{Z}_{(p)}}\twoheadrightarrow L_{1,\mathbb{Z}_{(p)}}$.
\begin{lemma}\label{I is torsor}
The scheme $I$ is a right $\mathcal {G}$-torsor over $\ES_K(G,X)$.
\end{lemma}

\begin{proof}
By \cite{CIMK}, the direct summand $L^\vee_{1,\mathbb{Z}_{(p)}}\subseteq L^\vee_{\mathbb{Z}_{(p)}}$ induces a direct summand $L^\vee_{1,\dr}\subseteq \V$. To prove the lemma, it suffices to prove that $L^\vee_{1,\dr}=\V_1$. But then it suffices to prove that $$L^\vee_{1,\dr}|_{\Sh_K(G,X)}=\V_1|_{\Sh_K(G,X)}.$$ As if we denote by $Grass_{\V}^{2g_1}$ the $\ES_K(G,X)$-scheme of locally direct summands of $\V$ with rank~$2g_1$. Then $Grass_{\V}^{2g_1}$ is proper over $\ES_K(G,X)$. The sub-bundles $L^\vee_{1,\dr}\subseteq \V$ and $\V_1\subseteq \V$ induce an $\ES_K(G,X)$-morphism $i:\ES_K(G,X)\rightarrow Grass_{\V}^{2g_1}\times Grass_{\V}^{2g_1}$. That $L^\vee_{1,\dr}|_{\Sh_K(G,X)}=\V_1|_{\Sh_K(G,X)}$ means that the restriction to $\Sh_K(G,X)$ of $i$ factors through the diagonal. But the diagonal is closed and $\ES_K(G,X)$ is reduced, so $i$ factors through the diagonal, which means that $L^\vee_{1,\dr}= \V_1$.

But $L^\vee_{1,\dr}|_{\Sh_K(G,X)}=\V_1|_{\Sh_K(G,X)}$ follows from the construction of these two bundles. We will follow \cite {CIMK} 2.2 and work with $L_{1,\dr}|_{\Sh_K(G,X)}$ and $\V_1^\vee|_{\Sh_K(G,X)}$. They are both closed subschemes of $\V^\vee|_{\Sh_K(G,X)}$, so to prove that they are equal, it suffices to prove that $$L_{1,\dr}|_{\Sh_K(G,X)}=L_{1,\dr}|_{\Sh_K(G,X)}+\V_1^\vee|_{\Sh_K(G,X)}= \V_1^\vee|_{\Sh_K(G,X)}.$$
But then one can pass to $\Sh_K(G,X)_{\mathbb{C}}$ and use descent. Let $\widetilde{\V_1^\vee|_{\Sh_K(G,X)}}$ be the pull back to $X\times G(\mathbb{A}_f)/K$ of $\V_1^\vee|_{\Sh_K(G,X)_\mathbb{C}}$. Then by the de Rham isomorphism, $\widetilde{\V_1^\vee|_{\Sh_K(G,X)}}$ equals to the vector bundle attached to the variation of Hodge structures given by $X$ and $G\rightarrow \text{GL}(V_1)$. Note that the quotient by $G(\mathbb{Q})$ of these two bundles give $L_{1,\dr}|_{\Sh_K(G,X)_\mathbb{C}}$ and $\V_1^\vee|_{\Sh_K(G,X)_\mathbb{C}}$ on $\Sh_K(G,X)_\mathbb{C}$ respectively, so $L_{1,\dr}|_{\Sh_K(G,X)_\mathbb{C}}=\V_1^\vee|_{\Sh_K(G,X)_\mathbb{C}}$.
\end{proof}

We write $i_1,i_2,p_1,p_2,i$ for the morphisms of the special fibers. To prove that the Ekedahl-Oort stratifications are independent of choices of symplectic embedding, it suffices to prove that the stratifications induced by $i_1$ and $i$ coincide. By Corollary \ref{inpe of tensors} and the proof of Lemma \ref{I is torsor}, the special fiber of $I$ is precisely the $G_0$-torsor in the $G_0$-zip over $\ES_{0,K}(G,X)$ constructed using $i$. Let us write $I$ for this special fiber and $(I,I_+,I_-,\iota)$ for the $G_0$-zip constructed using $i$. Let $(I_1,I_{1,+},I_{1,-},\iota_1)$ be the $G_0$-zip over $\ES_{0,K}(G,X)$ constructed using $i_1$.

There is a natural morphism $\epsilon:I\rightarrow I_1$ given by
$$f\in I(S)\mapsto f|_{L^\vee_{1,\kappa}\otimes O_S}\in I_1(S),\ \ \text{for all }\ES_{0,K}(G,X)\text{-scheme }S.$$
\begin{theorem}\label{inpe of symp}
The morphism $\epsilon$ induces an isomorphism $(I,I_+,I_-,\iota)\rightarrow(I_1,I_{1,+},I_{1,-},\iota_1)$ of $G_0$-zips. In particular, $i_1$ and $i$ give the same Ekedahl-Oort stratification.
\end{theorem}
\begin{proof}
The morphism $\epsilon:I\rightarrow I_1$ is clearly $G_0$-equivariant, and hence an isomorphism of $G_0$-torsors. For any $S/\ES_{0,K}(G,X)$, and any $f\in I_+(S)\subseteq I(S)$, $f$ maps the weight $1$ subspace of $L_\kappa^\vee\otimes O_S$ to $\mathrm{ker}(F)$, where $F$ is the Frobenius on $\V$. Let $F_1$ be the Frobenius on $\V_1$, then $\mathrm{ker}(F_1)=\mathrm{ker}(F)\cap \V_1$, as $\V_1\subseteq \V$ is induced by a morphism of abelian schemes and hence compatible with Frobenius. So $\epsilon(f)=f|_{L^\vee_{1,\kappa}\otimes O_S}$ maps the weight $1$ subspace of $L_{1,\kappa}^\vee\otimes O_S$ to $\mathrm{ker}(F)\cap \V_1=\mathrm{ker}(F_1)$, and hence lies in $I_{1,+}(S)$. But then $\epsilon|_{I_+}$ will automatically be an isomorphism of $P_+$-torsors. Similarly, $\epsilon|_{I_-}$ is an isomorphism of $P_-^{(p)}$-torsors.

Now we check the compatibility between $\iota$ and $\iota_1$. We first recall how $\iota$ and $\iota'$ are defined in Theorem \ref{G-zipES_0} 3). Let $\varphi_0$, $\varphi_1$ be as in Setting \ref{F-zip attached to abs}, and $\phi_0$, $\phi_1$ be as in Setting \ref{phi in standard G-zip}. Then $\iota:I_+^{(p)}/U_+^{(p)}\rightarrow I_-/U_-^{(p)}$ is the morphism
induced by
\begin{equation*}
\begin{split}
 I_+^{(p)}&\rightarrow I_-/U_-^{(p)}\\
f&\mapsto (\varphi_0\oplus \varphi_1)\circ\mathrm{gr}(f)\circ(\phi_0^{-1}\oplus \phi_1^{-1}), \forall\ S/\ES_0\text{ and }\forall\ f\in I_+^{(p)}(S).
 \end{split}
 \end{equation*}

We apply the constructions in Setting \ref{F-zip attached to abs} and Setting \ref{phi in standard G-zip} to $\V_1$ and $L^\vee_{1,\kappa}$ respectively, and denote the obtained morphisms by $\varphi'_0$, $\varphi'_1$, $\phi'_0$ and $\phi'_1$. Let $(L^\vee_{1,\kappa})^0$ (resp. $(L^\vee_{1,\kappa})^1$) be the subspace of $L^\vee_{1,\kappa}$ of weight $0$ (resp. 1) with respect to $\mu$, and $(L^\vee_{1,\kappa})_0$ (resp. $(L^\vee_{1,\kappa})_1$) be the subspace of $L^\vee_{1,\kappa}$ of weight $0$ (resp. 1) with respect to $\mu^{(p)}$. Then $\phi'_0$ and $\phi'_1$ are compatible with $\phi_0$ and $\phi_1$, in the sense that $$\phi_0|_{(L^\vee_{1,\kappa})^{(p)}/((L^\vee_{1,\kappa})^1)^{(p)}}=\phi'_0:(L^\vee_{1,\kappa})^{(p)}/((L^\vee_{1,\kappa})^1)^{(p)}\longrightarrow(L^\vee_{1,\kappa})_0$$ and
$$\phi_1|_{((L^\vee_{1,\kappa})^1)^{(p)}}=\phi'_1:((L^\vee_{1,\kappa})^1)^{(p)}\longrightarrow L^\vee_{0,\kappa}/(L^\vee_{1,\kappa})_0.$$
Let $F':\V_1^{(p)}\rightarrow \V_1$ and $V':\V_1\rightarrow \V_1^{(p)}$ be the Frobenius and Verschiebung on $\V_1$ respectively. Then
$V$ and $F$ are compatible with $V'$ and $F'$. This implies that
$$\varphi_0|_{(\V_1/\mathrm{ker}(F'\circ \delta))^{(p)}}=\varphi'_0:(\V_1/\mathrm{ker}(F'\circ \delta))^{(p)}\rightarrow \mathrm{ker}(F')$$
and $$\varphi_1|_{\mathrm{ker}(F'\circ \delta))^{(p)}}=\varphi'_1:(\mathrm{ker}(F'\circ \delta))^{(p)}\rightarrow \V_1/(\mathrm{im}(F')).$$

Then $\forall\ S/\ES_0\text{ and }\forall\ f\in I_+^{(p)}(S)$, we have
\begin{equation*}
\begin{split}
\iota'\circ\epsilon(f)=\iota'(f|_{L^\vee_{1,\kappa}\otimes O_S})=(\varphi'_0\oplus \varphi'_1)\circ\mathrm{gr}(f|_{L^\vee_{1,\kappa}\otimes O_S})\circ(\phi'^{-1}_0\oplus \phi'^{-1}_1)=\epsilon\circ\iota(f).
 \end{split}
 \end{equation*}
This shows that $\epsilon$ is an isomorphism of $G_0$-zips.
\end{proof}
\begin{remark}
The Ekedahl-Oort stratification does not depend on the choices of symplectic embeddings. So in particular, the theory of ordinariness is independent of symplectic embeddings. This coincides with the expectation that the variety $\ES_{0,K}(G,X)$ should have an interpretation as moduli space of ``abelian motives with $\mathcal {G}$-structure''. This moduli interpretation should be intrinsically determined by the Shimura datum, and hence independent of symplectic embeddings.
\end{remark}

\begin{remark}
A theory of Bruhat stratification has been defined and studied by Wedhorn in \cite{BruhatandFzip} (actually, we need the morphism $\zeta$ to define the Bruhat stratification on $\ES_{0,K}(G,X)$). In the case of Siegel modular varieties, the Bruhat stratification is precisely the $a$-number stratification. Theorem \ref{inpe of symp} also implies that the Bruhat stratification is independent of symplectic embeddings.
\end{remark}

\newpage

\section[Functoriality]{Functoriality}

Let $p$ be a prime bigger than $2$, and $(G,X)$ and $(G',X')$ be two Shimura data of Hodge type such that they both have good reduction at $p$. Let $E$ (resp. $E'$) be the reflex field of $(G,X)$ (resp. $(G',X')$). Let $K$ (resp. $K'$) be a compact open subgroup of $G(\mathbb{A}_f)$ (resp. $G'(\mathbb{A}_f)$) such that $K_p$ (resp. $K'_p$) is hyperspecial. Let $f:(G,X)\rightarrow(G',X')$ be a morphism of Shimura data, then $E\supseteq E'$. If $K$ and $K'$ are such that $f(K)\subseteq K'$, then $f$ induces a morphism of Shimura varieties $f:\Sh_K(G,X)\rightarrow \Sh_{K'}(G',X')_E$.

Let $v'$ be a place of $E'$ over $p$ with residue field ~$\kappa'$ and $v$ be a place of $E$ over $v'$ with residue field $\kappa$. Let $\ES_K(G,X)$ (resp. $\ES_{K'}(G',X')$) be the integral canonical model of $\Sh_K(G,X)$ (resp. $\Sh_{K'}(G',X')$). Then $f$ extends uniquely to a morphism $\ES_K(G,X)\rightarrow\ES_{K'}(G',X')_{O_{E,(v)}}$ whose special fiber $\ES_{0,K}(G,X)\rightarrow\ES_{0,K'}(G',X')_{\kappa}$ will still be denoted by $f$.

By ``functoriality'', we mean a certain kind of compatibility of Ekedahl-Oort stratifications with respect to $f$. But it seems that we need some extra assumptions. The reason is as follows. For a morphism $f:G_{\mathbb{Q}_p}\rightarrow G'_{\mathbb{Q}_p}$ such that $f(K_p)\subseteq K'_p$, it is NOT always possible to extend $f$ to a morphism $\mathcal {G}_1\rightarrow \mathcal {G}_2$ (see \cite{INTV} Proposition 3.1.2.1~b)). So there is NO natural morphism $G_0\rightarrow G'_0$, and hence there is NO direct way to compare $G_0$-zips and $G'_0$-zips.

\

\subsection[Basic settings]{Basic settings}

Let $\mathcal {G}/\mathbb{Z}_{(p)}$ (resp. $\mathcal {G}'/\mathbb{Z}_{(p)}$) be the reductive model of $G$ (resp. $G'$) with special fiber $G_0$ (resp. $G'_0$). Let $E$, $E'$, $\kappa$ and $\kappa'$ be as at the beginning of this section. Then by \cite{EOZ} Proposition 2.2.4, the Shimura datum $(G,X)$ (resp. $(G',X')$) determines a cocharacter $\mu$ (resp. $\mu'$) of $\mathcal {G}_{W(\kappa)}$ (resp. $\mathcal {G}'_{W(\kappa')}$) which is unique up to conjugacy. The reduction of $\mu$ (resp. $\mu'$) will still be denoted by $\mu$ (resp. $\mu'$).

Besides the conditions stated at the beginning of this section, we make the following assumption on $f:(G,X)\rightarrow(G',X')$.
\begin{asp}\label{nece asp}
There exists a morphism $\mathcal {G}_{\mathbb{Z}_{p}}\rightarrow \mathcal {G}'_{\mathbb{Z}_{p}}$ extending $f_{\mathbb{Q}_p}$. This morphism will be denoted by $\underline{f}$.
\end{asp}

\subsection[The morphism $\alpha$]{The morphism $\alpha$}\label{The morphism alpha}

The morphism $\underline{f}$ induces a natural morphism $\alpha:G_0\mathrm{-zip}^\mu_\kappa\rightarrow G'_0\mathrm{-zip}^{\mu'}_{\kappa'}\otimes \kappa$ which we will now explain. Still write $\mu$ for the cocharacter $\mathbb{G}_{m,\kappa}\rightarrow G_{0,\kappa}\rightarrow G'_{0,\kappa}$, then $\mu$ and $\mu'$ are $G'_0(\kappa)$-conjugate.

There is a natural morphism $\alpha_1:G_0\mathrm{-zip}^\mu_\kappa\rightarrow G'_0\mathrm{-zip}^{\mu}_{\kappa}$ as follows. The cocharacter $\mu$ induces homomorphisms $P_+\rightarrow P'_+$, $P_-\rightarrow P'_-$ and $L\rightarrow L'$. For any $\kappa$-scheme $S$ and any $S$-point $(I,I_+,I_-,\iota)$ of $\in G_0\mathrm{-zip}^\mu_\kappa$,
$$\alpha_1(I,I_+,I_-,\iota):=(I\times^{G_{0,S}} G'_{0,S},\ I_+\times^{P_{+,S}} P'_{+,S},\ I_-\times^{P_{-,S}^{(p)}} P'^{(p)}_{-,S},\ \iota'),$$
where $?_1\times^{?_2} ?_3$ is the quotient of $?_1\times ?_3$ equalizing the $?_2$-action on $?_1$ given by the torsor structure and that on $?_3$ induced by $\underline{f}$, and $\iota'$ is the composition of $L'^{(p)}$-equivariant isomorphisms

\begin{equation*}
\begin{split}
&(I_+\times^{P_{+,S}} P'_{+,S})^{(p)}/U'^{(p)}_{+}\stackrel{\simeq}{\longrightarrow}(I_+^{(p)}/U^{(p)}_{+})\times^{L_{S}^{(p)}} L'^{(p)}_{S},\\
&(I_+^{(p)}/U^{(p)}_{+})\times^{L_{S}^{(p)}} L'^{(p)}_{S}\stackrel{\iota\times \text{id}}{\longrightarrow}(I_-/U^{(p)}_{-})\times^{L_{S}^{(p)}} L'^{(p)}_{S},\\
\text{and }\ \ \ \ \ \ \ \ \ &(I_-/U^{(p)}_{-})\times^{L_{S}^{(p)}} L'^{(p)}_{S}\stackrel{\simeq}{\longrightarrow}(I_-\times^{P_{-,S}^{(p)}} P'^{(p)}_{-,S})/U'^{(p)}_-.
 \end{split}
 \end{equation*}

Let $\mu'\otimes 1$ be the base change to $\kappa$ of the cocharacter $\mu'$, then by \cite{muordinary} Remark 5.16 1), then there is an obvious isomorphism $\alpha_2:G'_0\mathrm{-zip}^{\mu'}_{\kappa'}\otimes \kappa\rightarrow G'_0\mathrm{-zip}^{\mu'\otimes 1}_{\kappa}$ given by base change. Let $g\in G'_0(\kappa)$ be such that $\text{int}(g)\circ (\mu'\otimes 1)=\mu$, then $g$ induces an isomorphism of algebraic stacks
\begin{equation*}
\begin{split}
\alpha_3^{g}:G'_0\mathrm{-zip}^{\mu'\otimes 1}_{\kappa}&\rightarrow G'_0\mathrm{-zip}^{\mu}_{\kappa}\\
(I,I_+, I_-,\iota)&\mapsto (I',I'_+, I'_-,\iota'):=(I,(I_+)\cdot g^{-1}, (I_-)\cdot \sigma(g)^{-1},r_{\sigma(g)^{-1}}\circ \iota\circ r_{\sigma(g)}),
 \end{split}
 \end{equation*}
where $r_{\sigma(g)}$ and $r_{\sigma(g)^{-1}}$ are the obvious morphisms $I'^{(p)}_+/U'^{(p)}_+\simeq I^{(p)}_+/U^{(p)}_+$ and $I_-/U^{(p)}_-\simeq I'_-/U'^{(p)}_-$ given by multiplication with $\sigma(g)$ and $\sigma(g)^{-1}$ on the right respectively.

\begin{remark}\label{unique alpha 3}
The morphism $\alpha_3^{g}$ is canonical, in the sense that it is uniquely determined by $\mu$ and $\mu'\otimes 1$ and does not depend on the choices of $g$. For an $h\in G'_0(\kappa)$ such that $\text{int}(h)\circ (\mu'\otimes 1)=\mu$, there exists an $l\in L'(\kappa)$, such that $h=gl$. Here $L'$ is, as before, the centralizer in $G'_{0,\kappa}$ of $\mu'$. Then
\begin{equation*}
\begin{split}
\alpha_3^{h}(I,I_+, I_-,\iota)&=(I,(I_+)\cdot h^{-1}, (I_-)\cdot \sigma(h)^{-1},r_{\sigma(h)^{-1}}\circ \iota\circ r_{\sigma(h)})\\
&=(I,(I_+)\cdot l^{-1}g^{-1}, (I_-)\cdot \sigma(l)^{-1}\sigma(g)^{-1},r_{\sigma(g)^{-1}}\circ r_{\sigma(l)^{-1}}\circ \iota\circ r_{\sigma(l)}\circ r_{\sigma(g)})\\
&=(I,(I_+)\cdot g^{-1}, (I_-)\cdot \sigma(g)^{-1},r_{\sigma(g)^{-1}}\circ \iota\circ r_{\sigma(g)}).
 \end{split}
 \end{equation*}
The last equality is because of that $I_+$ (resp. $I_-$) is $L'$ (resp. $L'^{(p)}$) stable and that $\iota$ is $L'^{(p)}$-equivariant. We will simply write $\alpha_3$ for $\alpha_3^{g}$, as it is independent of $g$.
\end{remark}

The morphism $\alpha$ is defined to be $\alpha_2^{-1}\circ\alpha_3^{-1}\circ \alpha_1.$

\subsection[Functoriality]{Functoriality}

We use the same notations as at the beginning of this section. Let $\zeta:\ES_{0,K}(G,X)\rightarrow G_0\mathrm{-zip}^\mu_\kappa$ and $\zeta':\ES_{0,K'}(G',X')\rightarrow G'_0\mathrm{-zip}^{\mu'}_{\kappa'}$. Moreover, we assume that the morphism of Shimura data $f:(G,X)\rightarrow (G',X')$ satisfies Assumption \ref{nece asp}.

By functoriality, we mean the following.
\begin{theorem}
The diagram
$$\xymatrix{
\ES_{0,K}(G,X)\ar[d]^{\zeta}\ar[r]^{f}& \ES_{0,K'}(G',X')_\kappa\ar[d]^{\zeta'\otimes\kappa} \\
G_0\mathrm{-zip}^\mu_\kappa\ar[r]^{\alpha} & G'_0\mathrm{-zip}^{\mu'}_{\kappa'}\otimes \kappa
}$$
is commutative.
\end{theorem}
\begin{proof}
The proof, which is a variation of that of Theorem \ref{inpe of symp}, will be divided into several steps.

\textbf{Step 1. }Let $i:(G,X)\rightarrow (\mathrm{GSp}(V,\psi),H)$ and $i':(G',X')\rightarrow (\mathrm{GSp}(V',\psi'),H')$ be symplectic embeddings. Note that we do NOT assume that there is any compatibility between $f$ and the symplectic embeddings. The weight cocharacter $w:\mathbb{G}_{m,\mathbb{Q}}\rightarrow G$ induces a $\mathbb{G}_{m,\mathbb{Q}}$-action on $V$ of weight 1, and the composition $f\circ w$ induces a $\mathbb{G}_{m,\mathbb{Q}}$-action on $V'$ of weight 1. Let $V_1=V\oplus V'$ and $\psi_1: V_1\times V_1\rightarrow \mathbb{Q}$ be such that
$$\psi_1\big((v,v'),(w,w')\big)=\psi(v,w)+\psi'(v',w'),\ \ \ \forall\ v,w\in V \ \mathrm{and}\ \forall\ v',w'\in V'.$$
Then $i$ and $i'\circ f$ induce a faithful representation of $G$ on $V_1$. Moreover, the argument in \ref{sum of ebd} shows that $G\subseteq \mathrm{GSp}(V_1,\psi_1)$ and this embedding induces an embedding of Shimura data $$i_1:(G,X)\subseteq (\mathrm{GSp}(V_1,\psi_1),H_1).$$

\textbf{Step 2. }There is a  $\mathbb{Z}$-lattice $L\subseteq V$ (resp. $L'\subseteq V'$) such that the Zariski closure of $G$ (resp. $G'$) in $\mathrm{GL}(L_{\mathbb{Z}_{(p)}})$ (resp. $\mathrm{GL}(L'_{\mathbb{Z}_{(p)}})$) $\mathcal {G}$ (resp. $\mathcal {G}'$) is reductive and such that $K_p=\mathcal {G}(\mathbb{Z}_p)$ (resp. $K'_p=\mathcal {G}'(\mathbb{Z}_p)$). Let $L_1$ be $L\oplus L'$. Consider the sequence of closed embedings
\begin{subeqn}\label{seq of closed sbgp'}
\mathcal {G}\times \mathcal {G}'\subseteq\mathrm{GL}(L_{\mathbb{Z}_{(p)}})\times \mathrm{GL}(L'_{\mathbb{Z}_{(p)}})\subseteq \mathrm{GL}(L_{1,\mathbb{Z}_{(p)}}).
\end{subeqn}
Let $\mathcal {G}_1$ be the Zariski closure of $G$ in $\mathrm{GL}(L_{1,\mathbb{Z}_{(p)}})$. Then $\mathcal {G}_1\subseteq\mathcal {G}\times \mathcal {G}'$. Flat base-change implies that $\mathcal {G}_{1,\mathbb{Z}_p}$ is the graph of $\underline{f}:\mathcal {G}_{\mathbb{Z}_p}\rightarrow \mathcal {G}'_{\mathbb{Z}_p}$. So $\mathcal {G}_1\cong \mathcal {G}$ and $\underline{f}$ is defined over $\mathbb{Z}_{(p)}$.

\textbf{Step 3. } Let $s\in L^\otimes_{\mathbb{Z}_{(p)}}$ (resp. $s'\in L'^\otimes_{\mathbb{Z}_{(p)}}$, $s_1\in L^\otimes_{1,\mathbb{Z}_{(p)}}$) be a tensor defining $\mathcal {G}\subseteq\mathrm{GL}(L_{\mathbb{Z}_{(p)}})$ (resp. $\mathcal {G}'\subseteq\mathrm{GL}(L'_{\mathbb{Z}_{(p)}})$, $\mathcal {G}\subseteq\mathrm{GL}(L_{1,\mathbb{Z}_{(p)}})$). Then $\mathcal {G}\times \mathcal {G}'\subseteq\mathrm{GL}(L_{1,\mathbb{Z}_{(p)}})$ is such that for all $\mathbb{Z}_{(p)}$-algebra $R$,
\begin{equation*}
\begin{split}
\mathcal {G}\times \mathcal {G}'(R)=\{g\in \mathrm{GL}&(L_{1,\mathbb{Z}_{(p)}})(R)\mid g(L\otimes R)=L\otimes R, g(L'\otimes R)=L'\otimes R\\
&\text{ and }g(s\otimes 1)=s\otimes 1, g(s'\otimes 1)=s'\otimes 1\}.
 \end{split}
 \end{equation*}
But then $\mathcal {G}$ is the group scheme such that for any $\mathbb{Z}_{(p)}$-algebra $R$,
\begin{equation*}
\begin{split}
\mathcal {G}(R)=&\{g\in \mathrm{GL}(L_{1,\mathbb{Z}_{(p)}})(R)\mid g(L\otimes R)=L\otimes R, g(L'\otimes R)=L'\otimes R\\
&\text{ and }g(s\otimes 1)=s\otimes 1, g(s'\otimes 1)=s'\otimes 1,g(s_1\otimes 1)=s_1\otimes 1\}.
 \end{split}
 \end{equation*}

\textbf{Step 4. }By our constructions in \ref{Comparing G_0-zips (II)}, the symplectic embeddings $i$, $i'$ and $i_1$ induce vector bundles $\V$, $\V'$ and $\V_1$ on $\ES_{0,K}(G,X)$.
The tensor $s$, $s'$ and $s_1$ induce tensors $\sdr\in \V$, $\sdr'\in \V'$ and $s_{1,\dr}\in \V_1$ respectively. Note that we have $\V_1=\V\oplus\V'$. Let $(I,I_+,I_-,\iota)$ be the $G_0$-zip on $\ES_{0,K}(G,X)$ constructed using $i$, and $(I_1,I_{1,+},I_{1,-},\iota_1)$ be the $G_0$-zip over $\ES_{0,K}(G,X)$ constructed using $i_1$. Then by Theorem \ref{inpe of symp}, $(I,I_+,I_-,\iota)\cong(I_1,I_{1,+},I_{1,-},\iota_1).$ We twist $(I_1,I_{1,+},I_{1,-},\iota_1)$ by $(G'_{0,\kappa}, \mu)$ using constructions at the beginning of \ref{The morphism alpha}, and get a $G'_0$-zip of type $\mu$ over $\ES_{0,K}(G,X)$ denoted by $(I'_1,I'_{1,+},I'_{1,-},\iota'_1)$.

\textbf{Step 5. }Let $(I',I'_+,I'_-,\iota')$ be the $G'_0$-zip of type $\mu'$ over $\ES_{0,K'}(G',X')$ constructed using $i'$. Let $(I',I'_+,I'_-,\iota')_\kappa$ be the pull back to $\ES_{0,K}(G,X)$ of $(I',I'_+,I'_-,\iota')$. The construction before Remark \ref{unique alpha 3}, $(I',I'_+,I'_-,\iota')_\kappa$ gives a $G'_0$-zip of type $\mu$ over $\ES_{0,K}(G,X)$ which will still be denoted by $(I',I'_+,I'_-,\iota')_\kappa$. Let $\epsilon:I'_1\rightarrow I'$ be the morphism given by restriction to $L^\vee_{2,\kappa}$. Then $\epsilon$ is an isomorphism of $G'_0$-torsors. By the proof of Theorem \ref{inpe of symp}, $\epsilon$ induces an isomorphism of $G'_0$-zips of type $\mu$. But this means that the diagram $$\xymatrix{
\ES_{0,K}(G,X)\ar[d]^{\zeta}\ar[r]^{f}& \ES_{0,K'}(G',X')_\kappa\ar[d]^{\zeta'\otimes\kappa} \\
G_0\mathrm{-zip}^\mu_\kappa\ar[r]^{\alpha} & G'_0\mathrm{-zip}^{\mu'}_{\kappa'}\otimes \kappa
}$$is commutative.
\end{proof}

\subsection[Basic examples]{Basic examples}

Here we give some basic examples where Assumption \ref{nece asp} is satisfied. 

\begin{example}Let $i:(G,X)\rightarrow (\mathrm{GSp}(V',\psi'),H')$ be a symplectic embedding. There exists a $\mathbb{Z}$-lattice $L'\subseteq V'$ such that the Zariski closure $\mathcal {G}$ of $G$ in $\mathrm{GL}(L'_{\mathbb{Z}_{(p)}})$ is reductive. The polarization $\psi'$ is not necessarily perfect pairing on $L'_{\mathbb{Z}_{(p)}}$. But by Zarhin's trick, we can take $L=(L'\oplus L'^\vee)^4$, then $\psi'$ induces a perfect paring on $L$ which will be denoted by $\psi$. Then the Zariski closure of $G$ in $\mathrm{GL}(L_{\mathbb{Z}_{(p)}})$ lies in $\mathrm{GSp}(L_{\mathbb{Z}_{(p)}},\psi)$, and hence there is an embedding $\mathcal {G}\hookrightarrow \mathrm{GSp}(L_{\mathbb{Z}_{(p)}},\psi)$. So there is a commutative diagram $$\xymatrix{
\ES_{0,K}(G,X)\ar[d]^{\zeta}\ar[r]& \ES_{0,K'}(\mathrm{GSp}(L_{\mathbb{Q}},\psi),H)_\kappa\ar[d]^{\zeta'\otimes\kappa} \\
G_0\mathrm{-zip}^\mu_\kappa\ar[r]& \mathrm{GSp}(L_{\mathbb{F}_p},\psi)\mathrm{-zip}^{\mu'}_{\mathbb{F}_p}\otimes \kappa.
}$$
\end{example}
\begin{example}
Let $(G,X)$ be a Shimura datum of PEL type with good reduction at $p$. Let $i:(G,X)\rightarrow (\mathrm{GSp}(V,\psi),H)$ be the tautological symplectic embedding. Then there is again a commutative diagram $$\xymatrix{
\ES_{0,K}(G,X)\ar[d]^{\zeta}\ar[r]& \ES_{0,K'}(\mathrm{GSp}(L_{\mathbb{Q}},\psi),H)_\kappa\ar[d]^{\zeta'\otimes\kappa} \\
G_0\mathrm{-zip}^\mu_\kappa\ar[r]& \mathrm{GSp}(L_{\mathbb{F}_p},\psi)\mathrm{-zip}^{\mu'}_{\mathbb{F}_p}\otimes \kappa.
}$$
\end{example}

\

\

Email: czhang\texttt{@}math.leidenuniv.nl \emph{or} zhangchao1217\texttt{@}gmail.com

\newpage
\addcontentsline{toc}{section}{References}
\begin{thebibliography}{Gro}
\bibitem{varideshi}Deligne, P.: \emph{Vari\`{e}t\'{e}s de Shimura: interpr\'{e}tation
modulaire, et techniques de construction de mod\`{e}les canoniques},
Automorphic forms, representations and $L$-functions, Proc. Sympos.
Pure math. XXXIII, PP. 247-289, Amer. Math. Soc., 1979.
\bibitem{CIMK}Kisin, M.: \emph{Integral models for Shimura varieties of abelian
type}, J. Amer. Math. Soc. 23, pp. 967-1012, 2010.
\bibitem{LRKisin}Kisin, M.: \emph{Mod $p$ points on Shimura varieties of abelian type}, preprint, avaiable online at
``http://www.math.harvard.edu/$\sim$kisin/dvifiles/lr.pdf''.
\bibitem{stackL}Laumon, G.; Moret-Bailly, L.: \emph{Champs alg\'{e}briques}, Ergebnisse der
Mathematik 39, Springer-Verlag, 2000.
\bibitem{introshv}Milne, J.: \emph{Introduction to Shimura
varieties}, Harmonic analysis, the trace formula, and Shimura
varieties, Clay Math. Proc. 4, pp. 265¨C378, Amer. Math. Soc., 2005.
\bibitem{disinv}Moonen, B; Wedhorn. T.: \emph{Discrete invariants of varieties in positive
characteristic}, Int. Math. Res. Not. 72, 3855-3903, 2004.
\bibitem{zipdata}Pink, R; Wedhorn, T; Ziegler, P.: \emph{Algebraic zip
data}, Documenta Math. 16, pp. 253-300, 2011.
\bibitem{zipaddi}Pink, R; Wedhorn, T; Ziegler, P.: \emph{$F$-zips with additional
structure}, arXiv:1208.3547, 2012.
\bibitem{INTV}Vasiu, A.: \emph{Integral canonical models of Shimura varieties of preabelian type},
Asian J. Math. 3 (1999), no. 2, 401-517
\bibitem{VW}Viehmann, E; Wedhorn, T.: \emph{Ekedahl-Oort and Newton strata for Shimura varieties of PEL
type}, Math. Ann. 356, pp. 1493-1550, 2013.
\bibitem{deRhamw}Wedhorn, T.: \emph{De Rham cohomology of varieties over fields of positive
characteristic}, Higher-dimensional geometry over finite fields, pp.
269-314, IOS Press, 2008.
\bibitem{BruhatandFzip}Wedhorn, T.: \emph{Bruhat strata and F-zips with additional structures},
available at http://arxiv.org/pdf/1302.6715.pdf.
\bibitem{muordinary}Wortmann, D.: \emph{The $\mu$-ordinary locus for Shimura varieties of Hodge type},
available at http://arxiv.org/pdf/1310.6444.pdf.
\bibitem{EOZ}Zhang, C.: \emph{Ekedahl-Oort strata for Shimura varieties of Hodge type}, arXiv:1312.4869, 2013.
\end{thebibliography}
\end{document}